\documentclass[12pt]{article}
\usepackage[margin=1in]{geometry}
\DeclareMathAlphabet\mathbfcal{OMS}{cmsy}{b}{n}
\usepackage[utf8]{inputenc}
\usepackage[dvipsnames]{xcolor}
\usepackage{mathtools}
\usepackage{enumerate}
\usepackage{amssymb} 
\usepackage[utf8]{inputenc}
\usepackage{amsfonts}
\usepackage{amsthm}
\usepackage{amsmath}
\usepackage{lscape}
\usepackage{mathrsfs}
\usepackage{graphicx}
\usepackage{tikz-cd}
\usepackage{calc}
\usepackage{float}
\usepackage{stmaryrd}
\usepackage{multirow}
\usepackage{lipsum}
\usepackage{amssymb}
\usepackage{mathrsfs}
\usepackage{mathtools}
\usepackage{mathdots}
\usepackage{fancyhdr}
\usepackage{tikz}
\usepackage{faktor}
\usepackage{setspace}
\usetikzlibrary{decorations.markings}
\newtheorem{theorem}{Theorem}
\theoremstyle{plain}
\usepackage{sectsty}
\usepackage{kpfonts}
\newtheorem{corollary}{Corollary}
\newtheorem{definition}{Definition}
\newtheorem{example}{Example}

\newtheorem{remark}{Remark}

\newtheorem{question}{Question}

\newtheorem*{proposition*}{Proposition}

\newcommand{\GL}{{\mbox{GL}}}
\newcommand{\Hom}{{\mbox{Hom}}}

\newcommand{\vol}{{\mbox{vol}}}

\newtheorem*{observation*}{Observation}
\newtheorem*{theorem*}{Theorem}
\newtheorem*{claim*}{Claim}  
\usepackage{blindtext}
\usepackage{fancyhdr}
\title{Geometric structures as variational objects, II}
\author{Gabriella Clemente}
\date{}
\allowdisplaybreaks
\makeatletter
\renewcommand\tableofcontents{%
    \@starttoc{toc}%
}
\begin{document}
\maketitle

\begin{abstract}
This note is the sequel of ``Geometric structures as variational objects, I." It generalizes the main result and perspectives of that work to a class of geometric structures that includes integrable almost-complex structures.
\end{abstract}

This note generalizes the material of \cite{CSCPS}, and \cite{GSVO1}. The object of study here is a class of differential geometric structures on a real, smooth manifold that are inspired by integrable almost-complex structures. The intention is to provide a new technical context based on a calculus of variations for vector bundle-valued forms that could lead to a different way of thinking about and doing almost-complex geometry.

\section{Compound geometric structures}
Let $M$ be a compact smooth manifold with $\dim_{\mathbb{R}}(M)=n.$ For any vector bundle $V\to M,$ let $\Omega^{\bullet}(M, V)=\bigoplus_{k\geq 0} \Omega^k(M,V)$ be the space of $V$-valued differential forms on $M.$ For future reference, any $\rho \in \Omega^{\bullet}(M, V)$ can be expressed as a sum of its homogeneous parts, $\rho=\sum_{k\geq 0} \rho_k.$ The projection onto $k$-th degree forms will be denoted by $p_k$; i.e.\ \[p_k:\Omega^{\bullet} (M, V) \to \Omega^k (M, V), \quad p_k (\rho)=\rho_k.\] 

Let $E, E',$ and $E''$ be real vector bundles on $M.$ Take any pair of bilinear vector bundle homomorphisms \[E' \xleftarrow{\phi} E\oplus E \xrightarrow{\psi} E'',\] where $E\oplus E\to M$ is the Whitney sum of $E \to M$ with itself, and where by bilinearity it is meant that, in particular, $\phi(\zeta_1+\zeta_2,\eta)=\phi(\zeta_1,\eta)+\phi(\zeta_2,\eta),$ $\phi(\zeta,\eta_1+\eta_2)=\phi(\zeta,\eta_1)+\phi(\zeta,\eta_2),$ $c\phi(\zeta,\eta)=\phi(c\zeta,\eta)=\phi(\zeta,c\eta),$ and likewise for $\psi.$ Denote the image of any $(\zeta,\eta) \in E \oplus E$ under $\phi$ and $\psi$ by $\zeta \cdot_{\phi} \eta$ and $\zeta \cdot_{\psi} \eta,$ respectively. Then, $\phi$ and $\psi$ give rise to maps \[\Omega^{\bullet}(M, E') \xleftarrow{\wedge_{\phi}} \Omega^{\bullet}(M, E)\otimes \Omega^{\bullet}(M, E) \xrightarrow{\wedge_{\psi}} \Omega^{\bullet}(M,E')',\] defined, for any \[\alpha \in \Omega^k(M,E), \beta \in \Omega^l(M,E),\] as the forms \[\alpha \wedge_{\phi} \beta \in \Omega^{k+l}(M,E'), \mbox{ and }\alpha \wedge_{\psi} \beta \in \Omega^{k+l}(M,E''),\] where 
\[(\alpha \wedge_{\phi} \beta)(X_1,\dots,X_{k+l})=\frac{1}{k!l!}\sum_{\sigma \in S_{k+l}} sign(\sigma)\alpha(X_{\sigma(1)},\dots,X_{\sigma(k)}) \cdot_{\phi} \beta(X_{\sigma(k+1)},\dots,X_{\sigma(k+l)}),\] and similarly
\[(\alpha \wedge_{\psi} \beta)(X_1,\dots,X_{k+l})=\frac{1}{k!l!}\sum_{\sigma \in S_{k+l}} sign(\sigma)\alpha(X_{\sigma(1)},\dots,X_{\sigma(k)}) \cdot_{\psi} \beta(X_{\sigma(k+1)},\dots,X_{\sigma(k+l)}).\]

The bilinear nature of $\phi,$ and $\psi$ implies that, for example, if $\gamma_1, \gamma_2 \in \Omega^k(M,E),$ $\beta \in \Omega^l(M,E),$ and $t$ is a constant, then \[(\gamma_1+t\gamma_2) \wedge_{\phi} \beta=\gamma_1 \wedge_{\phi} \beta+t\gamma_2 \wedge_{\phi} \beta,\] and a similar distributive property holds using linearity in the second argument of $\phi.$ And, of course, analogous equalities hold for $\wedge_{\psi}$ as well.

Observe that any such bilinear $f:E\oplus E\to E$ can be used to turn the space $\Omega^{\bullet}(M,E)$ into a graded algebra, namely, $\big(\Omega^{\bullet}(M,E),\wedge_f \big).$

Now, suppose that $\Omega^{\bullet}(M,E)$ is acted on left by $\Omega^{\bullet}(M,E'),$ and on the right by $\Omega^{\bullet}(M,E'').$ The right/left action notation will be unambiguously $\wedge$; i.e.\ if $\rho \in \Omega^{\bullet}(M,E),$ $\alpha \in \Omega^{\bullet}(M,E'),$ and $\beta \in \Omega^{\bullet}(M,E''),$ then $\alpha \wedge \rho \wedge \beta$ should be interpreted as $\rho$ being acted on the left by $\alpha,$ and on the right by $\beta.$ 

Let $\nabla^E$ be any linear connection on $E.$ Recall the exterior covariant derivative $d^{\nabla^E}$ associated with $\nabla^E,$ which is is the degree $1$ operator on $E$-valued forms, acting on $k$-forms $\alpha \in \Omega^k(M,E)$ as 

\begin{equation*}
\begin{split}
(d^{\nabla^E} \alpha)(\zeta_0,\dots,\zeta_k)&=\sum_{i=0}^k (-1)^i {\nabla^{E}}_{\zeta_i} \alpha(\zeta_0,\dots,\widehat{\zeta_i},\dots,\zeta_k)+\\
&\sum_{0 \leq i <j \leq k} (-1)^{i+j} \alpha ([\zeta_i,\zeta_j],\dots,\widehat{\zeta_i},\dots,\widehat{\zeta_j},\dots,\zeta_k).
\end{split}
\end{equation*}  

\begin{definition}\label{D1}
A form $\gamma \in \Omega^k(M,E)$ is called a compound geometric structure of degree $k$ if belongs in the kernel of a differential operator $P^{\nabla^E}_{(a,\phi,\psi)}$ on $\Omega^{\bullet}(M,E)$ of degree $1$ on $\Omega^k(M,E),$ and of the type
\begin{equation*}
\begin{split}
P^{\nabla^E}_{(a,\phi,\psi)}(\rho)&:=a_1 (\rho \wedge_{\phi} \rho) \wedge d^{\nabla^E} \rho \wedge (\rho \wedge_{\psi} \rho)+a_2 (\rho \wedge_{\phi} \rho) \wedge d^{\nabla^E} \rho+\\
&a_3  d^{\nabla^E} \rho \wedge (\rho \wedge_{\psi} \rho)+a_4 d^{\nabla^E} \rho,
\end{split}
\end{equation*}
where $a=(a_1,a_2,a_3,a_4) \in \mathbb{R}^4.$
Moreover, if \[q_{a_1}(l)=deg\big((\rho_l \wedge_{\phi} \rho_l) \wedge d^{\nabla^E} \rho_l \wedge (\rho_l \wedge_{\psi} \rho_l) \big),\] \[q_{a_2}(l)=deg\big((\rho_l \wedge_{\phi} \rho_l) \wedge d^{\nabla^E} \rho_l\big),\] \[q_{a_3}(l)=deg\big(d^{\nabla^E} \rho_l \wedge (\rho_l \wedge_{\psi} \rho_l)\big),\] then $P^{\nabla^E}_{(a,\phi,\psi)}$ is required to be such that $q_{a_i}(l) \neq k,$ for all $i=1,2,3,$ for all $l.$
\end{definition}

\begin{remark}\label{req1}
Let $\gamma \in \Omega^k(M,E).$ The degree $1$ requirement means that $P^{\nabla^E}_{(a,\phi,\psi)}(\gamma) \in \Omega^{k+1}(M,E),$ and so $q_{a_1}(k)=q_{a_2}(k)=q_{a_3}(k)=k+1.$ Indeed, $\gamma$ is a $k$-compound geometric structure if $P^{\nabla^E}_{(a,\phi,\psi)}(\gamma)=0.$
\end{remark}
 
\begin{remark}\label{practice}
In practice, these structures tend to live in some specific subset $U \subset \Omega^k(M,E),$ hence are describable as $U \cap \ker{P^{\nabla^E}_{(a,\phi,\psi)}}.$
\end{remark}

Of course, these geometric structures generalize in a rather obvious way to $E$-valued $k$-forms $\gamma$ such that 
\begin{equation*}
\begin{split}
0&=\sum_{i\geq 0} \big[c_i(\gamma \wedge_{\phi_i} \gamma) \wedge d^{\nabla^E} \gamma \wedge (\rho \wedge_{\psi_i} \rho)+c'_i(\gamma \wedge_{\phi_i} \gamma) \wedge d^{\nabla^E} \gamma+\\
&c''_i d^{\nabla^E} \gamma \wedge (\rho \wedge_{\psi_i} \rho)+c'''_i d^{\nabla^E} \gamma \big],
\end{split}
\end{equation*}
where all but finitely many of the coefficients $c_i,c'_i,c''_i,c'''_i \in \mathbb{R}$ are zero, and where $(\phi_i)_{i \geq 0}$ and $(\psi_i)_{i \geq 0}$ are sequences of bilinear maps $\phi_i: E\oplus E \to E',$ respectively $\psi_i: E\oplus E \to E''.$ Moreover, one could vary the right/left-actions within each term. These generalizations though technically intriguing, do not lead to any sort of recognizable structure of differential geometry. This is why the present note concerns itself exclusively with geometric structures as given in Definition \ref{D1}.

Nevertheless, compound geometric structures have a simple generalization that is geometrically significant.

\begin{definition}\label{alph}
Assume that $1 \leq p \leq n-k-1,$ and let $\alpha \in \Omega^p (M)$ be a non-trivial, closed differential form. An $\alpha$-compound geometric structure of degree $k$ is any element in the kernel of a differential operator $P^{\nabla^E}_{\alpha, (a,\phi,\psi)}:=\alpha \wedge P^{\nabla^E}_{(a,\phi,\psi)}$ on $\Omega^{\bullet}(M,E),$ where $P^{\nabla^E}_{(a,\phi,\psi)}$ is of the kind specified in Definition \ref{D1}. It will moreover be assumed that for all $l,$ and all $1 \leq i \leq3,$ $p \neq k-q_{a_i}(l).$ 
\end{definition}

Clearly, since 
\begin{equation*}
\begin{split}
&P^{\nabla^E}_{\alpha, (a,\phi,\psi)}(\gamma)(X_1,\dots,X_{k+p})=(\alpha \wedge P^{\nabla^E}_{(a,\phi,\psi)}(\gamma))(X_1,\dots,X_{k+p})\\
&=\frac{1}{p!k!}\sum_{\sigma \in S_{k+p}} sign(\sigma) \alpha(X_{\sigma(1)},\dots,X_{\sigma(p)})P^{\nabla^E}_{(a,\phi,\psi)}(\gamma)(X_{\sigma(p+1)},\dots,X_{\sigma(k+p)})
\end{split}
\end{equation*}

\begin{remark}\label{obv}
Compound geometric structures are $\alpha$-compound for all $\alpha.$
\end{remark}

\begin{example}{(Complex structures)}\label{only}
Assume that $M$ is almost-complex. All of what comes next is rephrasing the algebraic machinery discussed first in \cite{ACOTI}. Let $\nabla$ be any torsion free connection on $T_M.$ In the above construction, take \[E=T_M, \quad E''=\bigwedge^{\bullet} {T_M}=\bigoplus_{k\geq 0}\bigwedge^k {T_M}.\] Let $\psi:T_M \oplus T_M \to \bigwedge^{\bullet} {T_M}$ be the polyvector wedge product map $(u,v) \mapsto \psi(u,v)=u \wedge v \in \bigwedge^2 {T_M}.$ This induces the product \[\wedge_{\psi}:\Omega^{\bullet}(M,T_M) \otimes \Omega^{\bullet}(M,T_M) \to \Omega^{\bullet}\big(M, \bigwedge^{\bullet} {T_M}\big)\] that for any $\alpha \in \Omega^k(M,T_M),$ $\beta \in \Omega^l(M,T_M)$ is defined as
\[(\alpha \wedge_{\psi} \beta)(X_1,\dots,X_{k+l})=\frac{1}{k! l!}\sum_{\sigma \in S_{k+l}} sign(\sigma) \alpha(X_{\sigma(1)},\dots, X_{\sigma(k)}) \wedge \beta(X_{\sigma(k+1)},\dots,X_{\sigma(k+l)}).\] Recall the right $\Omega^{\bullet}\big(M, \bigwedge^{\bullet} {T_M}\big)$-action on $\Omega^{\bullet}(M,T_M),$ which for any $\rho \in \Omega^s(M,T_M),$ and $\gamma \in \Omega^i \big(M, \bigwedge^j {T_M}\big)$ is given as
\begin{equation*}
\begin{split}
(\rho \wedge \gamma)(X_1,\dots,X_{s-j+i}):=\begin{cases}
\frac{1}{(s-j)!i!} \sum_{\sigma \in S_{s-j+i}} sign(\sigma)\rho(X_{\sigma(1)},\dots,X_{\sigma(s-j)},\\
\cdot,\dots,\cdot) \big(\gamma(X_{\sigma(s-j+1)},\dots,X_{\sigma(s-j+i)})\big) & \mbox{ if }s \geq j \\
0 & \mbox{ if }s<j,
\end{cases}
\end{split}
\end{equation*}
where $\rho(X_{\sigma(1)},\dots,X_{\sigma(s-j)},\cdot,\dots,\cdot)$ is being viewed as an element of $\Omega^{s-j} \big(M, \Hom_{\mathbb{R}} \big(\bigwedge^j {T_M}, T_M\big)\big),$ while $\gamma(X_{\sigma(s-j+1)},\dots,X_{\sigma(s-j+i)}) \in \bigwedge^j {T_M}.$ That is possbile because \[\Omega^s(M,T_M) \subset \Omega^{s-j} \big(M,\Hom_{\mathbb{R}} \big(\bigwedge^j {T_M},T_M\big)\big).\] 

For the coefficients in the structure defining PDE, take $a_1=a_2=0,$ $a_2=1,$ $a_4=-1.$ Then, an almost-complex structure $A \in \Omega^1(M,T_M)$ is integrable iff \[P^{\nabla}_{(a,\psi)}(A)=d^{\nabla} A \wedge (A \wedge_{\psi} A)-d^{\nabla} A=0\] (Lemma 1, \cite{ACOTI}). So, complex structures are examples of degree $1$ compound geometric structures. They may be collectively described as \[AC(M) \cap \ker{P^{\nabla}_{(a,\psi)}},\] where \[AC(M):=\{J \in \Omega^1(M,T_M) \mid J^2=-Id\}\] is the space of almost-complex structures on $M.$
\end{example}

\begin{example}{($\alpha$-integrable almost-complex structures)}\label{only1}
Let $\alpha \in \Omega^1(M)$ be non-trivial and $d$-closed. Building on the previous example, notice that an almost-complex structure $A$ is $\alpha$-integrable if $P^{\nabla}_{\alpha,(a,\psi)}(A)=0$ (cf.\ \cite{CSCPS} and Definition \ref{alph}). 
\end{example}

\section{Variational realizations}
As explained in \cite{GSVO1}, $\Omega^{\bullet}(M,E)$ carries an $L^2$-inner product. Let $h_E$ be a bundle metric on $E,$ and $g$ be a Riemannian metric on $M.$ Assume that $rk(E)=m,$ and let $(\epsilon_i)_{i=1}^m$ be a local frame for $E.$ Then, the bilinear map \[\wedge_E:\Omega^k(M,E) \otimes \Omega^l(M,E) \to \Omega^{k+l}(M),\] where if \[\alpha=\sum_{i=1}^m a_i \otimes \epsilon_i \in \Omega^k(M,E), \mbox{ and } \beta=\sum_{j=1}^m b_j \otimes \epsilon_j \in \Omega^l(M,E),\] \[\alpha \wedge_E \beta=\sum_{i,j=1}^m q_i \wedge b_j h_E(\epsilon_i,\epsilon_j),\] along with the pairing \[\langle \cdot,\cdot\rangle_{h_E}: \Omega^k(M,E) \otimes \Omega^k(M,E) \to \mathbb{R},\] where if \[\alpha'=\sum_{j=1}^m a'_j \otimes \epsilon_j, \mbox{ then } \langle\alpha, \alpha'\rangle_{h_E}=\sum_{i,j=1}^m \langle a_i, a'_j\rangle h_E(\epsilon_i,\epsilon_j),\] allow for an extension of the Hodge star operator to $\Omega^{\bullet}(M,E),$ which has the defining property that \[\star_{h_E}:  \Omega^k(M,E) \to \Omega^{n-k}(M,E), \quad \alpha \wedge_{h_E} \star_{h_E} \alpha'=\langle \alpha,\alpha'\rangle_{h_E} \vol_g.\] And then, through the pairing \[\langle\cdot,\cdot\rangle_k^{h_E}:\Omega^k(M,E) \otimes \Omega^k(M,E) \to \mathbb{R}, \quad \langle \alpha, \alpha' \rangle_k^{h_E}:=\int_M \alpha \wedge_{h_E} \star_{h_E} \alpha',\] one obtains the $L^2$-product \[\langle\langle \cdot, \cdot \rangle\rangle_{h_E}:\Omega^{\bullet}(M,E) \otimes \Omega^{\bullet}(M,E) \to \mathbb{R}, \quad \langle\langle A,B \rangle\rangle_{h_E}:=\sum_{k\geq 0} \langle A_k, B_k \rangle_k^{h_E},\] where $A=\sum_{k\geq 0} A_k,$ $B=\sum_{k\geq 0} B_k,$ and $A_k, B_k \in \Omega^k(M,E).$

Let $0 \leq m_1 < m_2 <\dots <m_s$ be any finite increasing sequence of non-negative integers. Given any such sequence, one can redefine the covariant exterior derivative w.r.t.\ $\nabla^E$ so that it vanishes in degrees $m_1,m_2,\dots,m_s,$ while remaining intact in all other degrees. Explicitly, if \[\delta_i^j=\begin{cases}
1 & \mbox{ if } i=j\\
0 & \mbox{ if } i\neq j
\end{cases}\] is the Kronecker delta function, and 
\begin{definition}\label{kronos}
If the restriction $d^{\nabla^E}\Big|_{\Omega^k (M, E)}$ is denoted by $d_k^{\nabla^E}:\Omega^k (M, E) \to \Omega^{k+1} (M, E),$ define an operator on $\Omega^{\bullet} (M, E)$ by \[d^{\nabla^E}[m_1,m_2,\dots,m_s]:=\sum_{i=0}^{n-1}(1-\sum_{j=1}^s\delta_{m_j-1}^i)d_i^{\nabla^E}\] so that 
\[d^{\nabla^E}[m_1,m_2,\dots,m_s]\Big|_{\Omega^i (M, E)}= \begin{cases} 
      d_i^{\nabla^E}, & i \neq m_1-1,m_2-1,\dots,m_s-1 \\
      0 , & i = m_1-1,m_2-1,\dots,m_s-1.
   \end{cases}
\]
\end{definition}
The original operator $d^{\nabla^E}$ has a formal adjoint $\delta^{\nabla^E}$ w.r.t.\ the above described $L^2$-inner product. Denoting $\delta^{\nabla^E}\Big|_{\Omega^k (M, E)}$ by $\delta_k^{\nabla^E}:\Omega^k (M, E) \to \Omega^{k-1} (M, E),$ the formal adjoint of $d^{\nabla^E}[m_1,m_2,\dots,m_s]$ satisfies
\[\delta^{\nabla^E}[m_1,m_2,\dots,m_s]\Big|_{\Omega^i (M, T_M)}= \begin{cases} 
      \delta_i^{\nabla^E}, & i \neq m_1,m_2,\dots,m_s\\
      0 , & i = m_1,m_2,\dots,m_s.
   \end{cases}
\]

For example, when $s=1,$ this coincides with the operator $d^{\nabla^E}[k]$ from \cite{GSVO1}. An instance of this operator with $s=2$ appears in \cite{CSCPS}

\begin{remark}\label{etde}
A differential operator of the kind $P^{\nabla^E}_{(a,\phi,\psi)}$ can be redefined with $d^{\nabla^E}[m_1,m_2,\dots,m_s]$:
\begin{equation*}
\begin{split}
P^{\nabla^E}_{(a,\phi,\psi)}[m_1,m_2,\dots,m_s](\rho)&:=a_1 (\rho \wedge_{\phi} \rho) \wedge d^{\nabla^E}[m_1,m_2,\dots,m_s] \rho \wedge (\rho \wedge_{\psi} \rho)+\\
&a_2 (\rho \wedge_{\phi} \rho) \wedge d^{\nabla^E}[m_1,m_2,\dots,m_s] \rho+\\
&a_3  d^{\nabla^E}[m_1,m_2,\dots,m_s] \rho \wedge (\rho \wedge_{\psi} \rho)+a_4 d^{\nabla^E}[m_1,m_2,\dots,m_s] \rho.
\end{split}
\end{equation*}
\end{remark}
Recall the definition of $k$\emph{-variational object} \cite{GSVO1}. This is the degree $k$ projection of any critical point of any functional with domain contained in some space of differential forms, taking values in a vector bundle.
\begin{theorem}\label{T1}
Compound $k$-geometric structures are $k$-variational objects. They can be realized via the functional \[\mathcal{P}^{\nabla^E}_{(a,\phi,\psi)}[k]:\Omega^{\bullet}(M,E)\to \mathbb{R}, \quad \mathcal{P}^{\nabla^E}_{(a,\phi,\psi)}[k](\gamma):=\langle\langle \mathbf{P}^{\nabla^E}_{(a,\phi,\psi)}[k](\gamma),\gamma\rangle\rangle_{h_E},\] where $\mathbf{P}^{\nabla^E}_{(a,\phi,\psi)}[k](\gamma):=\sum_{r \geq 0} P^{\nabla^E}_{(a,\phi,\psi)}[k](\gamma_r).$
\end{theorem}

\begin{proof}
Suppose that the geometric structures that one is seeking to variationally realize are given as $U \cap \ker{P^{\nabla^E}_{(a,\phi,\psi)}}.$ The goal is to show functional $\mathcal{P}^{\nabla^E}_{(a,\phi,\psi)}[k]$ restricted to a sub-domain that results from modifying $\gamma \in \Omega^{\bullet}(M,E)$ in degrees other thank $k$ has the key feature that its set of critical points, when projected onto $\Omega^k(M,E),$ matches precisely with $U \cap \ker{P^{\nabla^E}_{(a,\phi,\psi)}}.$ 

The 1st variation of the functional is
\begin{equation*}
\begin{split}
\frac{d}{dt} \Big|_{t=0} \mathcal{P}^{\nabla^E}_{(a,\phi,\psi)}[k](\gamma+t\beta)&=\big\langle\big\langle \sum_{r \geq 0} P^{\nabla^E}_{(a,\phi,\psi)}[k](\gamma_r+t\beta_r), \gamma_r+t\beta_r \big\rangle\big\rangle_{h_E}\\
&=\big\langle\big\langle\sum_{r\geq 0} \big[a_1 \big(\beta_r \wedge_{\phi} \gamma_r+\gamma_r \wedge_{\phi} \beta_r\big) \wedge d^{\nabla^E}[k] \gamma_r \wedge (\gamma_r \wedge_{\psi} \gamma_r)+\\
&a_1 (\gamma_r \wedge_{\phi} \gamma_r) \wedge d^{\nabla^E}[k] \beta_r \wedge (\gamma_r \wedge_{\psi} \gamma_r)+\\
&a_1 (\gamma_r \wedge_{\phi} \gamma_r) d^{\nabla^E}[k] \gamma_r \wedge (\beta_r \wedge_{\psi} \gamma_r+\gamma_r \wedge_{\psi} \beta_r)+\\
&a_2 (\beta_r \wedge_{\phi} \gamma_r+\gamma_r \wedge_{\phi} \beta_r) \wedge d^{\nabla^E}[k] \gamma_r+a_2 (\gamma_r \wedge_{\phi} \gamma_r) \wedge d^{\nabla^E}[k] \beta_r+\\
&a_3 d^{\nabla^E}[k] \beta_r \wedge (\gamma_r \wedge_{\psi} \gamma_r) + a_3 d^{\nabla^E}[k] \gamma_r \wedge (\beta_r \wedge_{\psi} \gamma_r+\gamma_r \wedge_{\psi} \beta_r)+\\
&a_4 d^{\nabla^E}[k] \beta_r\big], \gamma \big\rangle\big\rangle_{h_E}+\big\langle\big\langle \beta, \mathbf{P}^{\nabla^E}_{(a,\phi,\psi)}[k](\gamma)\big\rangle\big\rangle_{h_E}.
\end{split}
\end{equation*}
Consider the linear maps
\[L_{1,\gamma_r}:\Omega^r(M,E) \to \Omega^{q_{a_1}(r)}(M,E),\]
\[L'_{1,\gamma_r}:\Omega^{r+1}(M,E) \to \Omega^{q_{a_1}(r)}(M,E),\]
\[L''_{1,\gamma_r}:\Omega^r(M,E) \to \Omega^{q_{a_1}(r)}(M,E),\]
\[L_{2,\gamma_r}:\Omega^r(M,E) \to \Omega^{q_{a_2}(r)}(M,E),\]
\[L'_{2,\gamma_r}:\Omega^{r+1}(M,E) \to \Omega^{q_{a_2}(r)}(M,E),\]
\[L_{3,\gamma_r}:\Omega^{r+1}(M,E) \to \Omega^{q_{a_3}(r)}(M,E),\] and
\[L'_{3,\gamma_r}:\Omega^r(M,E) \to \Omega^{q_{a_3}(r)}(M,E)\] with definitions
\[L_{1,\gamma_r}(b)=a_1 \big(b \wedge_{\phi} \gamma_r+\gamma_r \wedge_{\phi} b\big) \wedge d^{\nabla^E}[k] \gamma_r \wedge (\gamma_r \wedge_{\psi} \gamma_r),\]
\[L'_{1,\gamma_r}(B)=a_1 (\gamma_r \wedge_{\phi} \gamma_r) \wedge B \wedge (\gamma_r \wedge_{\psi} \gamma_r),\]
\[L''_{1,\gamma_r}(b)=a_1 (\gamma_r \wedge_{\phi} \gamma_r) d^{\nabla^E}[k] \gamma_r \wedge (b \wedge_{\psi} \gamma_r+\gamma_r \wedge_{\psi} b),\]
\[L_{2,\gamma_r}(b)=a_2 (b \wedge_{\phi} \gamma_r+\gamma_r \wedge_{\phi}b) \wedge d^{\nabla^E}[k] \gamma_r,\]
\[L'_{2,\gamma_r}(B)=a_2 (\gamma_r \wedge_{\phi} \gamma_r) \wedge B,\]
\[L_{3,\gamma_r}(B)=a_3 B \wedge (\gamma_r \wedge_{\psi} \gamma_r),\] and
\[L'_{3,\gamma_r}(b)=a_3 d^{\nabla^E}[k] \gamma_r \wedge (b \wedge_{\psi} \gamma_r+\gamma_r \wedge_{\psi} b).\] It follows that
\begin{equation*}
\begin{split}
\frac{d}{dt} \Big|_{t=0} \mathcal{P}^{\nabla^E}_{(a,\phi,\psi)}[k](\gamma+t\beta)&=\sum_{r\geq 0} \big[\big\langle L_{1,\gamma_r}(\beta_r)+L'_{1,\gamma_r}(d^{\nabla^E}[k] \beta_r)+L''_{1,\gamma_r}(\beta_r),\gamma_{q_{a_1}(r)}\big\rangle^{h_E}_{q_{a_1}(r)}+\\&\big\langle L_{2,\gamma_r} (\beta_r)+L'_{2,\gamma_r}(d^{\nabla^E}[k] \beta_r),\gamma_{q_{a_2}(r)}\big\rangle^{h_E}_{q_{a_2}(r)}+\\
&\big\langle L_{3,\gamma_r}(d^{\nabla^E}[k] \beta_r)+L'_{3,\gamma_r}(\beta_r),\gamma_{q_{a_3}(r)}\big\rangle^{h_E}_{q_{a_3}(r)}\\
&\big\langle a_4 d^{\nabla^E}[k] \beta_r, \gamma_{r+1}\big\rangle^{h_E}_{r+1}+\big\langle\big\langle \beta, \mathbf{P}^{\nabla^E}_{(a,\phi,\psi)}[k](\gamma)\big\rangle\big\rangle_{h_E}.
\end{split}
\end{equation*}
Let \[(L_{1,\gamma_r})^*:\Omega^{q_{a_1}(r)}(M,E) \to \Omega^r(M,E)\] be the adjoint of $L_{1,\gamma_r},$ and designate the linear adjoints of all remaining maps with $^*$ as well. Then, 
\begin{equation*}
\begin{split}
\frac{d}{dt} \Big|_{t=0} \mathcal{P}^{\nabla^E}_{(a,\phi,\psi)}[k](\gamma+t\beta)&=\sum_{r\geq 0} \big[\big\langle \beta_r, (L_{1,\gamma_r})^*(\gamma_{q_{a_1}(r)})+\delta^{\nabla^E}[k]\big((L'_{1,\gamma_r})^*(\gamma_{q_{a_1}(r)})\big)+(L''_{1,\gamma_r})^*(\gamma_{q_{a_1}(r)})+\\
&(L_{2,\gamma_r})^*(\gamma_{q_{a_2}(r)})+\delta^{\nabla^E}[k]\big((L'_{2,\gamma_r})^*(\gamma_{q_{a_2}(r)})\big)+\delta^{\nabla^E}[k]\big((L_{3,\gamma_r})^*(\gamma_{q_{a_3}(r)})\big)+\\
&(L'_{3,\gamma_r})^*(\gamma_{q_{a_3}(r)})+a_4\delta^{\nabla^E}[k] \gamma_{r+1}\big\rangle^{h_E}_r\big]+\big\langle\big\langle \beta, \mathbf{P}^{\nabla^E}_{(a,\phi,\psi)}[k](\gamma)\big\rangle\big\rangle_{h_E}\\
&=\big\langle\big\langle \beta, \sum_{r\geq 0} \big[(L_{1,\gamma_r})^*(\gamma_{q_{a_1}(r)})+\delta^{\nabla^E}[k]\big((L'_{1,\gamma_r})^*(\gamma_{q_{a_1}(r)})\big)+(L''_{1,\gamma_r})^*(\gamma_{q_{a_1}(r)})+\\
&(L_{2,\gamma_r})^*(\gamma_{q_{a_2}(r)})+\delta^{\nabla^E}[k]\big((L'_{2,\gamma_r})^*(\gamma_{q_{a_2}(r)})\big)+\delta^{\nabla^E}[k]\big((L_{3,\gamma_r})^*(\gamma_{q_{a_3}(r)})\big)+\\
&(L'_{3,\gamma_r})^*(\gamma_{q_{a_3}(r)})+a_4\delta^{\nabla^E}[k] \gamma_{r+1}\big]+\mathbf{P}^{\nabla^E}_{(a,\phi,\psi)}[k](\gamma)\big\rangle\big\rangle_{h_E}.
\end{split}
\end{equation*}

Now, simply read off the parts of the Euler-Lagrange (EL) equation that involve $\gamma_k$ (excluding the structural equation). The first line in the above calculation says that (almost) all of those parts occur in the homogeneous degree $k$ stage of the EL equation:  
\begin{equation*}
\begin{split}
&(L_{1,\gamma_k})^*(\gamma_{q_{a_1}(k)})+\delta^{\nabla^E}[k]\big((L'_{1,\gamma_k})^*(\gamma_{q_{a_1}(k)})\big)+(L''_{1,\gamma_k})^*(\gamma_{q_{a_1}(k)})+(L_{2,\gamma_k})^*(\gamma_{q_{a_2}(k)})+\\
&\delta^{\nabla^E}[k]\big((L'_{2,\gamma_k})^*(\gamma_{q_{a_2}(k)})\big)+\delta^{\nabla^E}[k]\big((L_{3,\gamma_k})^*(\gamma_{q_{a_3}(k)})\big)+(L'_{3,\gamma_k})^*(\gamma_{q_{a_3}(k)})+\\
&a_4\delta^{\nabla^E}[k] \gamma_{k+1}+a_4d^{\nabla^E}[k] \gamma_{k-1}=0.
\end{split}
\end{equation*}
Since $q_{a_i}(k)=k+1,$ for all $1 \leq i \leq 3,$ it is possible to take $\gamma_{k+1}=0.$ 

There is one more place where $\gamma_k$ appears: in $a_4\delta^{\nabla^E}[k] \gamma_{k},$ which is part of the degree $k-1$ part of the EL equation. Had one chosen to work with the genuine covariant exterior derivative, this would have been a worrisome term. However, since the functional has been defined with $d^{\nabla^E}[k],$ $\delta^{\nabla^E}[k] \gamma_{k}=0,$ in this way undoing the problem without touching $\gamma_k.$ What this procedure does is free $\gamma_k$ of unnecessary constraints.

And now look at the degree $k+1$ part of the EL equation, which is where the structural equation appears,
\begin{equation*}
\begin{split}
&P^{\nabla^E}_{(a,\phi,\psi)}[k](\gamma_k)+(L_{1,\gamma_{k+1}})^*(\gamma_{q_{a_1}(k+1)})+\delta^{\nabla^E}[k]\big((L'_{1,\gamma_{k+1}})^*(\gamma_{q_{a_1}(k+1)})\big)+(L''_{1,\gamma_{k+1}})^*(\gamma_{q_{a_1}(k+1)})+\\&(L_{2,\gamma_{k+1}})^*(\gamma_{q_{a_2}(k+1)})+\delta^{\nabla^E}[k]\big((L'_{2,\gamma_{k+1}k})^*(\gamma_{q_{a_2}(k+1)})\big)+\delta^{\nabla^E}[k]\big((L_{3,\gamma_{k+1}})^*(\gamma_{q_{a_3}(k+1)})\big)+\\&(L'_{3,\gamma_{k+1}})^*(\gamma_{q_{a_3}(k+1)})+a_4\delta^{\nabla^E}[k] \gamma_{k+2}=0.
\end{split}
\end{equation*}

This suggests taking $\gamma_{q_{a_i}(k+1)}=0,$ for all $1 \leq i \leq 3,$ and keep in mind that $q_{a_i}(k+1)\neq k,$ and also $\delta^{\nabla^E} \gamma_{k+2}=0.$ 

Let $\mathcal{S}$ be the critical point set of the functional. Then, if \[\widetilde{\Omega}^k(M,E):=\{\gamma \in \Omega^{\bullet}(M,E) \mid \gamma_{k+1}=0, \gamma_{q_{a_1}(k+1)}=0, \gamma_{q_{a_2}(k+1)}=0, \gamma_{q_{a_3}(k+1)}=0, \delta^{\nabla^E} \gamma_{k+2}=0\},\] the critical point set of the restriction of $ \mathcal{P}^{\nabla^E}_{(a,\phi,\psi)}[k]$ to \[\widetilde{U}:=\{\gamma \in \widetilde{\Omega}^k(M,E) \mid \gamma_k \in U\}\] equals $\mathcal{S} \cap \widetilde{U},$ and set theoretically, \[p_k(\mathcal{S} \cap \widetilde{U})\] coincides with the compound geometric structures of interest. 
\end{proof}

A set such as $\widetilde{\Omega}^k(M,E)$ will be referred to as an \emph{intermediary sub-domain}.

\begin{theorem}\label{T2}
Compound $k$-$\alpha$-geometric structures are $k$-variational objects. They can be realized via the functional \[\mathcal{P}^{\nabla^E}_{\alpha,(a,\phi,\psi)}[k]:\Omega^{\bullet}(M,E)\to \mathbb{R}, \quad \mathcal{P}^{\nabla^E}_{\alpha,(a,\phi,\psi)}[k](\gamma):=\langle\langle \mathbf{P}^{\nabla^E}_{\alpha, (a,\phi,\psi)}[k](\gamma),\gamma\rangle\rangle_{h_E},\] where $\mathbf{P}^{\nabla^E}_{\alpha, (a,\phi,\psi)}[k](\gamma):=\sum_{r \geq 0} P^{\nabla^E}_{\alpha, (a,\phi,\psi)}[k](\gamma_r).$
\end{theorem}
\begin{proof}
Again, it will be assumed that the structures are given as $U \cap \ker{P^{\nabla^E}_{\alpha, (a,\phi,\psi)}},$ for some subset $U \subset \Omega^k(M, E),$ and also that $p$ is even, without any loss. Building on the proof of Theorem \ref{T1}, consider the maps \[L_{\alpha, 1,\gamma_r}:=\alpha \wedge L_{1,\gamma_r}, L'_{\alpha, 1,\gamma_r}:=\alpha \wedge L'_{1,\gamma_r}, L'_{\alpha, 1,\gamma_r}:=\alpha \wedge L'_{1,\gamma_r},\]
\[L_{\alpha, 2,\gamma_r}:=\alpha \wedge L_{2,\gamma_r}, L'_{\alpha, 2,\gamma_r}:=\alpha \wedge L'_{2,\gamma_r},\] and 
\[L_{\alpha, 3,\gamma_r}:=\alpha \wedge L_{3,\gamma_r}, L'_{\alpha, 3,\gamma_r}:=\alpha \wedge L'_{3,\gamma_r},\] and \[N^{\alpha}_r:\Omega^r (M, E) \to \Omega^{r+p} (M, E), \quad N^{\alpha}_r(x)=a_4\alpha \wedge x.\] Let their linear adjoints be designated with the symbol $^*$; e.g.\ $(L_{\alpha, 1,\gamma_r})^*:\Omega^{q_{a_1}(r)+p}(M,E)\to\Omega^r(M,E).$

The first variation is
\begin{equation*}
\begin{split}
\frac{d}{dt} \Big|_{t=0} \mathcal{P}^{\nabla^E}_{\alpha,(a,\phi,\psi)}[k](\gamma+t\beta)&=\big\langle\big\langle \beta, \sum_{r\geq 0} \big[(L_{\alpha, 1,\gamma_r})^*(\gamma_{q_{a_1}(r)+p})+\\
&\delta^{\nabla^E}[k]\big((L'_{\alpha, 1,\gamma_r})^*(\gamma_{q_{a_1}(r)+p})\big)+(L''_{\alpha, 1,\gamma_r})^*(\gamma_{q_{a_1}(r)+p})+\\
&(L_{\alpha, 2,\gamma_r})^*(\gamma_{q_{a_2}(r)+p})+\delta^{\nabla^E}[k]\big((L'_{\alpha, 2,\gamma_r})^*(\gamma_{q_{a_2}(r)+p})\big)+\\
&\delta^{\nabla^E}[k]\big((L_{\alpha, 3,\gamma_r})^*(\gamma_{q_{a_3}(r)+p})\big)+(L'_{\alpha, 3,\gamma_r})^*(\gamma_{q_{a_3}(r)+p})+\\
&(N^{\alpha}_r)^*(\delta^{\nabla^E}[k]\gamma_{p+r+1})\big]+\mathbf{P}^{\nabla^E}_{\alpha, (a,\phi,\psi)}(\gamma)\big\rangle\big\rangle_{h_E} 
\end{split}
\end{equation*}

The degree $k$ part of the EL equation is 
\begin{equation*}
\begin{split}
&(L_{\alpha, 1,\gamma_k})^*(\gamma_{k+1+p})+\delta^{\nabla^E}[k]\big((L'_{\alpha, 1,\gamma_k})^*(\gamma_{k+1+p})\big)+(L''_{\alpha, 1,\gamma_k})^*(\gamma_{k+1+p})+(L_{\alpha, 2,\gamma_k})^*(\gamma_{k+1+p})+\\
&\delta^{\nabla^E}[k]\big((L'_{\alpha, 2,\gamma_k})^*(\gamma_{k+1+p})\big)+\delta^{\nabla^E}[k]\big((L_{\alpha, 3,\gamma_k})^*(\gamma_{k+1+p})\big)+(L'_{\alpha, 3,\gamma_k})^*(\gamma_{k+1+p})+\\
&(N^{\alpha}_k)^*(\delta^{\nabla^E}[k]\gamma_{k+1+p})+a_4 \alpha \wedge d^{\nabla^E}[k] \gamma_{k-p-1}=0.
\end{split}
\end{equation*}
It so happens that it contains all occurrences of $\gamma_k,$ beyond those within the structural equation, with the exception of $(N^{\alpha}_{k-1-p})^*(\delta^{\nabla^E}[k]\gamma_k).$ However, the latter is automatically zero, and therefore does not count. So it suffices to take $\gamma_{k+1+p}=0.$

The degree $k+1+p$ part of the EL equation is
\begin{equation*}
\begin{split}
&P^{\nabla^E}_{\alpha, (a,\phi,\psi)}[k](\gamma_k)+(L_{\alpha, 1,\gamma_{k+1+p}})^*(\gamma_{q_{a_1}(k+1+p)+p})+\delta^{\nabla^E}[k]\big((L'_{\alpha, 1,\gamma_{k+1+p}})^*(\gamma_{q_{a_1}(k+1+p)+p})\big)+\\
&(L''_{\alpha, 1,\gamma_{k+1+p}})^*(\gamma_{q_{a_1}(k+1+p)+p})+(L_{\alpha, 2,\gamma_{k+1+p}})^*(\gamma_{q_{a_2}(k+1+p)+p})+\delta^{\nabla^E}[k]\big((L'_{\alpha, 2,\gamma_{k+1+p}})^*(\gamma_{q_{a_2}(k+1+p)+p})\big)+\\
&\delta^{\nabla^E}[k]\big((L_{\alpha, 3,\gamma_{k+1+p}})^*(\gamma_{q_{a_3}(k+1+p)+p})\big)+(L'_{\alpha, 3,\gamma_{k+1+p}})^*(\gamma_{q_{a_3}(k+1+p)+p})+\\
&(N^{\alpha}_{k+1+p})^*(\delta^{\nabla^E}[k]\gamma_{2p+k+2})=0.
\end{split}
\end{equation*}
Therefore, take $\gamma_{q_{a_i}(k+1+p)+p}=0,$ for all $1 \leq i \leq 3,$ and $\delta^{\nabla^E} \gamma_{2p+k+2}=0.$ Let the intermediary sub-domain be given as 
\begin{equation*}
\begin{split}
\widetilde{\Omega}^k(M,E)&:=\{\gamma \in \Omega^{\bullet}(M,E) \mid \gamma_{k+1+p}=0, \gamma_{q_{a_1}(k+1+p)+p}=0, \gamma_{q_{a_2}(k+1+p)+p}=0,\\
&\gamma_{q_{a_3}(k+1+p)+p}=0, \delta^{\nabla^E} \gamma_{2p+k+2}=0\}
\end{split}
\end{equation*}
Let $\mathcal{S}$ be the critical point set of the functional. Then, its restriction to \[\widetilde{U}:=\{\gamma \in \widetilde{\Omega}^k(M,E) \mid \gamma_k \in U\}\] has critical point set $\mathcal{S} \cap \widetilde{U},$ and so \[p_k(\mathcal{S} \cap \widetilde{U})\] coincides with the compound $\alpha$-geometric structures.
\end{proof}

\begin{remark}\label{also}
The functionals \[\mathcal{P}^{\nabla^E}_{(a,\phi,\psi)}[k, k+2] \mbox{ and } \mathcal{P}^{\nabla^E}_{\alpha,(a,\phi,\psi)}[k, 2p+k+2]\] on $\Omega^{\bullet}(M,E)$ also realize compound $k$-geometric, respectively $\alpha$-compound-$k$-geometric structures. They manage to relax their intermediary sub-domains to \[\widetilde{\Omega}^k(M,E):=\{\gamma \in \Omega^{\bullet}(M,E) \mid \gamma_{k+1}=0, \gamma_{q_{a_1}(k+1)}=0, \gamma_{q_{a_2}(k+1)}=0, \gamma_{q_{a_3}(k+1)}=0\},\] respectively
\begin{equation*}
\begin{split}
\widetilde{\Omega}^k(M,E)&:=\{\gamma \in \Omega^{\bullet}(M,E) \mid \gamma_{k+1+p}=0, \gamma_{q_{a_1}(k+1+p)+p}=0, \gamma_{q_{a_2}(k+1+p)+p}=0,\\
&\gamma_{q_{a_3}(k+1+p)+p}=0\}.
\end{split}
\end{equation*}
\end{remark}

\begin{corollary}\label{fgfof}
Let \[(L_{1, \gamma})^*(\gamma)=\sum_{r\geq 0} (L_{1,\gamma_r})^*(\gamma_{q_{a_1}(r)}), (L'_{1, \gamma})^*(\gamma)=\sum_{r\geq 0}(L'_{1,\gamma_r})^*(\gamma_{q_{a_1}(r)}), (L''_{1, \gamma})^*(\gamma)=\sum_{r\geq 0}(L''_{1,\gamma_r})^*(\gamma_{q_{a_1}(r)}),\] \[(L_{2, \gamma})^*(\gamma)=\sum_{r\geq 0} (\gamma_{q_{a_2}(r)}), (L_{2, \gamma})^*(\gamma)=\sum_{r\geq 0} (L'_{2,\gamma_r})^*(\gamma_{q_{a_2}(r)}),\] and \[(L_{3, \gamma})^*(\gamma)=\sum_{r\geq 0}(L_{3,\gamma_r})^*(\gamma_{q_{a_3}(r)}), (L'_{3, \gamma})^*(\gamma)=\sum_{r\geq 0}(L'_{3,\gamma_r})^*(\gamma_{q_{a_3}(r)}).\]
The formal gradient flow of $\mathcal{P}^{\nabla^E}_{(a,\phi,\psi)}[k]$ is

\begin{equation*}
\begin{split}
\frac{\partial \gamma}{\partial t}&=-\Big((L_{1, \gamma})^*(\gamma)+\delta^{\nabla^E}[k]\big((L'_{1, \gamma})^*(\gamma)\big)+(L''_{1, \gamma})^*(\gamma)+(L_{2, \gamma})^*(\gamma)+\\
&\delta^{\nabla^E}[k]\big((L_{2, \gamma})^*(\gamma)\big)+\delta^{\nabla^E}[k]\big((L_{3, \gamma})^*(\gamma) \big)+(L'_{3, \gamma})^*(\gamma)+\\
&a_4\delta^{\nabla^E}[k]\gamma+\mathbf{P}^{\nabla^E}_{(a,\phi,\psi)}[k](\gamma)\Big).
\end{split}
\end{equation*}

Let \[(L_{\alpha, 1, \gamma})^*(\gamma)=\sum_{r\geq 0} (L_{\alpha, 1,\gamma_r})^*(\gamma_{q_{a_1}(r)+p}), (L'_{\alpha, 1, \gamma})^*(\gamma)=\sum_{r\geq 0}(L'_{\alpha, 1,\gamma_r})^*(\gamma_{q_{a_1}(r)+p}),\] \[(L''_{\alpha, 1, \gamma})^*(\gamma)=\sum_{r\geq 0}(L''_{\alpha, 1,\gamma_r})^*(\gamma_{q_{a_1}(r)+p}),\] \[(L_{\alpha, 2, \gamma})^*(\gamma)=\sum_{r\geq 0} (L_{\alpha, 2,\gamma_r})^*(\gamma_{q_{a_2}(r)+p}), (L'_{\alpha, 2, \gamma})^*(\gamma)=\sum_{r\geq 0}(L'_{\alpha, 2,\gamma_r})^*(\gamma_{q_{a_2}(r)+p}),\] and \[(L_{3, \gamma})^*(\gamma)=\sum_{r\geq 0}(L_{\alpha, 3,\gamma_r})^*(\gamma_{q_{a_3}(r)+p}), (L'_{\alpha, 3, \gamma})^*(\gamma)=\sum_{r\geq 0}(L'_{\alpha, 3,\gamma_r})^*(\gamma_{q_{a_3}(r)+p}),\] and let \[(N^{\alpha})^*(\delta^{\nabla^E}[k] \gamma)=\sum_{r\geq 0} (N^{\alpha}_r)^*(\delta^{\nabla^E}[k]\gamma_{p+r+1}).\]

The formal gradient flow of $\mathcal{P}^{\nabla^E}_{\alpha,(a,\phi,\psi)}[k]$ is

\begin{equation*}
\begin{split}
\frac{\partial \gamma}{\partial t}&=-\Big((L_{\alpha, 1, \gamma})^*(\gamma)+\delta^{\nabla^E}[k]\big((L'_{\alpha, 1, \gamma})^*(\gamma)\big)+(L''_{\alpha, 1, \gamma})^*(\gamma)+(L_{\alpha, 2, \gamma})^*(\gamma)+\\
&\delta^{\nabla^E}[k]\big((L_{\alpha, 2, \gamma})^*(\gamma)\big)+\delta^{\nabla^E}[k]\big((L_{\alpha, 3, \gamma})^*(\gamma) \big)+(L'_{\alpha, 3, \gamma})^*(\gamma)+\\
&(N^{\alpha})^*(\delta^{\nabla^E}[k] \gamma)+\mathbf{P}^{\nabla^E}_{\alpha, (a,\phi,\psi)}(\gamma)\Big).
\end{split}
\end{equation*}
\end{corollary}

The formal gradient flows for \[\mathcal{P}^{\nabla^E}_{(a,\phi,\psi)}[k, k+2] \mbox{ and } \mathcal{P}^{\nabla^E}_{\alpha,(a,\phi,\psi)}[k, 2p+k+2]\] can be obtained via a similar process. Observe how Theorems \ref{T1} and \ref{T2} generalize Theorems 2 and 3 from \cite{CSCPS}.

\section{Stability}
In general, the main challenge seems to be to\\

\noindent
\textbf{Geometric structure existence (GSE) problem:} Obtain a full characterization of (compact) smooth manifolds that admit a geometric structure of a specified kind.\\

In the realm of almost-complex manifolds, it makes more sense to state the challenge from the vantage point of non-existence since at the moment, there are apparently no known examples of almost-complex, non-complex manifolds of real dimension at least $6.$

It could be useful to try to understand the extent to which the functionals studied here are (non-)convex, and their EL equations are (non-)elliptic. It could be meaningful as well to try establishing the short-time existence of their associated flows. However, this section seeks to postulate a formal analog of K-stability by using the asymptotic derivative definition, which is briefly summarized below. A concrete example of the sought formal analogy can be found in Section 4 of \cite{CSCPS}. K-stability has been shown to be equivalent to the existence of a Fano K{\"a}hler-Einstein (KE) metric \cite{CDS}, and this is one reason to believe that perhaps the concept is more widely applicable to existence questions of differential geometry.

Recall that if $X$ is a Fano manifold, and $\Omega$ is a fixed K{\"a}hler class, then the K-energy is the unique functional $\mathcal{K}$ on $\Omega$ whose critical points are the constant scalar curvature K{\"a}hler (cscK) metrics, and such that $\mathcal{K}(0)=0.$ Let $L \rightarrow X$ be an ample line bundle; i.e.\ for some $r,$ there is a basis of sections in $H^0(X,L^r),$ giving an embedding $\epsilon:X \hookrightarrow \mathbb{CP}^{N_r},$ $\epsilon(p)=[s_0(p):\dots:s_{N_r}(p)].$ A \emph{test-configuration} is the data $(\epsilon, \lambda),$ where $\lambda$ is a $\mathbb{C}^*$-action on $\mathbb{CP}^{N_r},$ or equivalently, $\lambda:\mathbb{C}^* \hookrightarrow \GL_{N_r +1}(\mathbb{C})$ is a 1-parameter subgroup. The \emph{central fiber} is the flat limit $X_0:=\lim_{t \to 0} \lambda(t) \cdot X.$ For all $t \neq 0,$ $\omega_t:=\frac{1}{t} \epsilon^* (\omega_{FS})$ is a K{\"a}hler metric on $\lambda(t) \cdot X.$ Up to normalization, the DF invariant of the test-configuration $(\epsilon,\lambda)$ is the asymptotic derivative \[DF(\epsilon,\lambda):=\lim_{t \to \infty} \frac{d\mathcal{K}(\omega_t)}{dt}.\] Then, K-stability is the condition that the DF invariant stay positive along all non-trivial test-configurations. 

The following is a sketch of how the above could be implemented in the context of compound geometric structures. Say the GSE Problem was about some geometric structures on a manifold $M$ that are generally describable as $U_M \cap \ker{P^{\nabla^{E_M}}_{(a,\phi_M,\psi_M)}}$ for a subspace $U_M \subset \Omega^k(M,E_M).$ Let $(M, \gamma)$ be a \emph{pre-structured space}, meaning that $\gamma \in U_M.$ Suppose that there existed a structured space $(N, \Gamma),$ $\Gamma \in U_N \cap \ker{P^{\nabla^{E_N}}_{(a',\phi_N,\psi_N)}},$ along with a universal embedding $F:(M,\gamma) \hookrightarrow (N,\Gamma),$ canonically inducing $\gamma$ as a pre-structure on $F(M).$ For example, depending on the kind of structures one is dealing with, $F$ could induce $\gamma$ via some reasonably defined pull-back. In the almost-complex case (cf.\ Section 4 \cite{CSCPS}), where the pre-structures are indeed almost-complex structures, the pull-back structure is obtained from a universal, transverse to a distribution, totally real embedding into a complex manifold, $F:(M,J) \hookrightarrow (Z,J_Z,\mathcal{D}).$ The complex structure on $Z$ descends to a complex structure $J_Z:{T_Z}/{\mathcal{D}} \to {T_Z}/{\mathcal{D}},$ which when conjugated by a point-wise defined $\mathbb{R}$-linear isomorphism ${T_Z}/{\mathcal{D}} \simeq_g F_{*}(T_M),$ gives an almost-complex structure on $F(M).$ For details on this embedding theory, see \cite{GUES} and the sources cited in there.

Let $\theta$ be a $\GL^{+}_1(\mathbb{R})$-action on $N,$ where $\GL^{+}_1(\mathbb{R})$ is the identity component. Call the pair $(F,\theta)_{\gamma}$ a \emph{test-configuration of} M \emph{relative (rel.)} $\gamma.$ If $M_t:=\theta(t) \cdot M,$ and $t > 0,$ the diffeomorphism $f_t:M \to M_t,$ $f_t(p)=\theta(t)p$ gives a pre-structure $\gamma_t:=(f_t)_{*} \circ F^* \Gamma \circ (f_t)^{-1}_{*}$ on $M_t.$ Using the terminology introduced in the previous section, choose an extension $\rho_{\gamma_t} \in \widetilde{U_{M_t}}.$ Say that $M$ is\\

\noindent
\emph{stable} rel.\ $\gamma$ iff for all test-configurations $(F,\theta)_{\gamma}$ rel.\ $\gamma,$ \[P\big((F,\theta)_{\gamma}\big):=\lim_{t \to \infty} \frac{d\mathcal{P}^{\nabla^{E_M}}_{(a,\phi,\psi)}[k](\rho_{\gamma_t})}{dt} > 0,\]

\noindent
$\alpha$-\emph{stable} rel.\ $\gamma$ iff for all test-configurations $(F,\theta)_{\gamma}$ rel.\ $\gamma,$ \[P_{\alpha} \big((F,\theta)_{\gamma}\big):=\lim_{t \to \infty} \frac{d\mathcal{P}^{\nabla^{E_M}}_{\alpha,(a,\phi,\psi)}[k](\rho_{\gamma_t})}{dt} > 0.\]

It seems natural to ask:

\begin{question}\label{K1}
Is there a link between stability and the GSE Problem?
\end{question}

As they stand, these stability definitions are devoid of geometrical content. In the Fano K{\"a}hler setting, the geometry comes from a correspondence between cssK metrics on $X$ in $\Omega$ and zeros of a moment map for a (complexified) Hamiltonian action of the symplectomorphism group on the space of compatible almost-complex structures, viewed as an infinite dimensional K{\"a}hler manifold \cite{Don}. Then, the Kempf-Ness Theorem furnishes a finite-dimensional picture that draws a parallelism between stability in the geometric invariant theory sense (the Hilbert-Mumford criterion) and the existence of canonical K{\"a}hler metrics \cite{Gabor}. So, unlike K-stability, the stability notions postulated here are not supported by any sort of general algebro-geometric framework that could infuse them with meaning. This raises the 

\begin{question}\label{A1}
Is stability truly an algebraic phenomenon? 
\end{question}

Or could it be an overarching moduli space phenomenon that, in the Fano KE instance, happens to be more easily detectable through algebraic means?

\noindent
Gabriella Clemente                                

\noindent
e-mail: clemente6171@gmail.com
\end{document}